\documentclass[11pt]{amsart}
\usepackage{float}
\usepackage[usenames]{color}
\usepackage{graphicx}
\usepackage{amssymb}
\usepackage{amscd}
\theoremstyle{plain}
\newtheorem{thm}{Theorem}[section]
\newtheorem{lemma}[thm]{Lemma}

\newtheorem{prop}[thm]{Proposition}

\theoremstyle{definition}
\newtheorem{defi}[thm]{Definition}

\newtheorem{rmk}[thm]{Remark}

\def\dim{\mathop{\hbox {dim}}\nolimits}

\def\det{\mathop{\hbox {det}}\nolimits}

\def\Hom{\mathop{\hbox {Hom}}\nolimits}

\def\Ind{\mathop{\hbox{Ind}}\nolimits}

\newcommand{\uInd}{^\mathrm{u}\mspace{-3mu}\mbox{Ind}}

\newcommand{\frb}{\mathfrak{b}}

\newcommand{\frg}{\mathfrak{g}}
\newcommand{\frh}{\mathfrak{h}}

\newcommand{\frk}{\mathfrak{k}}
\newcommand{\frl}{\mathfrak{l}}

\newcommand{\frp}{\mathfrak{p}}

\newcommand{\frt}{\mathfrak{t}}

\newcommand{\frgl}{\mathfrak{gl}}

\newcommand{\bbC}{\mathbb{C}}

\newcommand{\bbN}{\mathbb{N}}

\newcommand{\bbR}{\mathbb{R}}

\newcommand{\bbZ}{\mathbb{Z}}

\newcommand{\C}{\mathbb{C}}

\newcommand{\Z}{\mathbb{Z}}

\newcommand{\K}{\mathbb{K}}

\newcommand{\GL}{{\mathrm{GL}}}
\newcommand{\SL}{{\mathrm{SL}}}

\newcommand{\GU}{{\mathrm{GU}}}

\newcommand{\SU}{{\mathrm{SU}}}

\newcommand{\rmH}{{\mathrm{H}}}

\newcommand{\CS}{\mathcal{S}}

\newcommand{\be}{\begin {equation}}
\newcommand{\ee}{\end {equation}}
\newcommand{\bp}{\begin {proof}}
\newcommand{\ep}{\end {proof}}

\newcommand{\od}{\operatorname{d}}
\begin{document}

\title[the nonvanishing hypothesis for $\GL_n(\bbC)\times \GL_n(\bbC)$]
{On the nonvanishing hypothesis for Rankin-Selberg convolutions for $\GL_n(\bbC)\times \GL_n(\bbC)$}

\author{Chao-Ping Dong}
\address[Dong]{Institute of Mathematics, Hunan University, Changsha 410082,
P.~R.~China}
\email{chaoping@hnu.edu.cn}
\thanks{Dong is supported by NSFC grant 11571097 and the China Scholarship Council.}

\author{Huajian Xue}
\address[Xue]{Academy of Mathematics and Systems Science,
Chinese Academy of Sciences\\
Beijing, 100190, P.~R.~China} \email{xuehuajian12@mails.ucas.ac.cn}

\abstract{Inspired by Sun's breakthrough in establishing the nonvanishing hypothesis for Rankin-Selberg convolutions for the groups $\GL_n (\bbR)\times \GL_{n-1} (\bbR)$ and  $\GL_n (\bbC)\times \GL_{n-1} (\bbC)$, we confirm it for $\GL_{n} (\bbC)\times \GL_n (\bbC)$ at the central critical point.}
\endabstract



\maketitle

\section{Introduction}\label{sec-intro}

The nonvanishing hypothesis is vital to the arithmetic study of critical values of higher degree L-functions and to the constructions of higher degree p-adic L-functions. Recently, Sun made a breakthrough by confirming it for $\GL_n (\bbR)\times \GL_{n-1} (\bbR)$ and  $\GL_n (\bbC)\times \GL_{n-1} (\bbC)$, see \cite{Sun}. The current paper aims to consider the $\GL_{n} (\bbC)\times \GL_n (\bbC)$ case, which has been expected by  Greni\'e since 2003, see page 284 of \cite{Gre}.

Fix an integer $n\geq 2$. Let $\K$ be a topological field which is isomorphic to $\C$, and write
\[
\iota_1, \iota_2: \K \to \C
\]
for two distinct isomorphisms.

Let $B_n(\C) = T_n(\C) U_n(\C)$ be the group of upper triangular matrices in
$\GL_n(\C)$. Here $T_n(\C)$ is the standard maximal torus in $\GL_n(\C)$, while $U_n(\C)$ is the
standard unipotent radical of $B_n(\C)$.
We identify $\Z^n$ with the set of algebraic characters of $T_n(\C)$ by sending  $\mu = (\mu_i)$ to $t \mapsto \prod_{i} t_i^{\mu_i}$.

Fix a sequence of integers
\[
\mu=(\mu_1\geq\mu_2\geq\cdots\geq\mu_n;\, \mu_{n+1}\geq\mu_{n+2}\geq\cdots\geq\mu_{2n}).
\]
Denote by $F_\mu$ the irreducible algebraic representation of $\GL_n(\C)\times\GL_n(\C)$ with highest weight $\mu$. It is also viewed as an irreducible representation of the real Lie group $\GL_n(\K)$ via the complexification map
\be \label{eq: comlexify}
\GL_n(\K) \to \GL_n(\C)\times\GL_n(\C), \qquad g \mapsto(\iota_1(g), \iota_2(g)).
\ee
As usual, we do not distinguish a representation with its underlying space.

Recall that a representation of a real reductive
group is called a \emph{Casselman-Wallach representation} if it is smooth, Fr\'echet,
of moderate growth, and its Harish-Chandra module has finite length, see \cite{Cas} and Chapter 11 of \cite{Wa2} for more details.
Denote by $\Omega (\mu)$ the set of isomorphism classes of irreducible Casselman-Wallach representations $\pi$ of $\GL_n(\K)$ such that
\begin{itemize}
\item[$\bullet$] $\pi\vert_{\SL_n(\K)}$ is unitarizable and tempered; and
\item[$\bullet$] the relative Lie algebra cohomology
\be
\rmH^\bullet(\frgl_n(\C)\oplus\frgl_n(\C), \GU(n); \pi\otimes F_\mu^{\vee})\neq 0,
\ee
\end{itemize}
where $\frgl_n(\C)\oplus\frgl_n(\C)$ is viewed as the complexification of $\frgl_n(\K)$ through the differential of \eqref{eq: comlexify}, and
\[
\GU(n):=\{g\in\GL_n(\K)\vert\iota_1(g)\iota_2(g)^{\mathrm{t}}\mbox{ is a saclar matrix}\}.
\]

According to Section 3 of \cite{Clo}, we have
\be\label{card-Omega-mu}
\#\Omega (\mu)=\begin{cases}
0, & \mbox{ if } \mu \mbox{ is not pure}; \\
1, & \mbox{ if } \mu \mbox{ is pure}.
\end{cases}
\ee
Here ``$\mu$ is pure'' means that there is an integer  $w$ such that
\be\label{mu-pure-w}
\mu_1+\mu_{2n}=\mu_2+\mu_{2n-1}=\cdots=\mu_n+\mu_{n+1}=w.
\ee
In such a case, we shall say that $\mu$ is \emph{pure with weight} $w$.
Assume that $\mu$ is pure, and let $\pi_\mu$ be the unique representation in $\Omega (\mu)$.

Put
\be\label{bn}
b_n:=\frac{n(n-1)}{2}.
\ee
Then by Lemma 3.14 of \cite{Clo}, we have
$$
\rmH^{b}(\frgl_n(\C)\oplus\frgl_n(\C), \GU(n); \pi_\mu\otimes F_\mu^{\vee})= 0,
\quad \mbox{if } b<b_n,
$$
and
$$
\mathrm{dim}\,\rmH^{b_n}(\frgl_n(\C)\oplus\frgl_n(\C), \GU(n); \pi_\mu\otimes F_\mu^{\vee})= 1.
$$

We fix another sequence of integers
\[
\nu=(\nu_1\geq\nu_2\geq\cdots\geq\nu_n;\, \nu_{n+1}\geq\nu_{n+2}\geq\cdots\geq\nu_{2n}).
\]
Assume that $\nu$ is pure with weight $w^{\prime}$.
Define $F_\nu$, $\Omega (\nu)$ and $\pi_\nu$ similarly.

Write
\[
G:=\GL_n(\K) \qquad \mbox{and} \qquad \tilde K:=\GU(n).
\]
Let $P$ be the standard maximal parabolic subgroup of $G$ of type $(n-1, 1)$. Let ``$\vert \mspace{-1mu}\cdot \mspace{-1mu}\vert_\K$'' denotes the normalized absolute value of $\K$. That is, $\vert z\vert_\K=\iota_1(z) \iota_2(z)$ for $z\in \K$. Let
\[
H_s(p):=\delta_P^s(p)\cdot \eta(p_{nn})^{-1}=\vert \mspace{-3mu}\det p \vert_\K^s \cdot \vert p_{nn}\vert^{-ns}_\K\cdot \eta(p_{nn})^{-1},
\]
where $\eta=\omega_{\pi_{\mu}}\omega_{\pi_{\nu}}$ is the product of the central characters of $\pi_{\mu}$ and $\pi_{\nu}$. For $s\in \C$, define the normalized smooth induced representation
\be\label{Ij}
I_s:=\Ind_P^G \left(H_{s-\frac 1 2}\right).
\ee

Write
\[
G^3:=G\times G\times G \qquad \mbox{and} \qquad \tilde K^3:=\tilde K\times \tilde K\times\tilde K,
\]
then $G$ (resp. $\tilde K$) embeds in $G^3$ (resp. $\tilde K^3$) diagonally. Here and henceforth, we use the corresponding lower case gothic letter to indicate the complexified Lie algebra of a Lie group.

Assume that $\mu$ and $\nu$ are compatible (see Definition \ref{def-compatible}), and that $j\in \bbZ$ is a critical place (see Section \ref{sec-cri.pl.}) for $\pi_\mu \times \pi_\nu$. Denote
\[
\pi_\xi:=\pi_\mu \hat{\otimes} \pi_\nu \hat{\otimes} I_j   \qquad (\mbox{the completed projective space product}).
\]
By Section 2 of \cite{Jac}, the Rankin-Selberg integrals produce a nonzero element
\be
\phi_\pi \in \Hom_G(\pi_\xi, \bbC).
\ee

Let $V_j$ be the finite dimensional representation of $\GL_n(\bbC)\times \GL_n(\bbC)$ defined in Proposition \ref{prop-compatible} (b) or (c),
and put
\be\label{cn}
c_n=n-1.
\ee
Then in the setting of Proposition \ref{prop-BW} (b) and (c), we have
\[
\rmH^{c}(\frg, \tilde K; I_j \otimes V_j)= 0, \quad \mbox{if } c<c_n,
\]
and
\[
\dim \rmH^{c_n}(\frg, \tilde K; I_j \otimes V_j)=1.
\]

Denote
\be\label{F-xi-check}
F_\xi^{\vee}:= F_{\mu}^{\vee} \otimes F_{\nu}^{\vee} \otimes V_j.
\ee
Then there is a nonzero element
\be
\phi_F\in \Hom_{\GL_n(\bbC)\times \GL_n(\bbC)}(F_\xi^\vee, \bbC).
\ee

By K\"{u}nneth formula,
\[
\dim \rmH^{2b_n+c_n}(\frg^3, \tilde K^3; \pi_\xi \otimes F_\xi^\vee)=1.
\]
Note that
\[
\dim_\C(\frg/\tilde\frk)=n^2-1=2b_n+c_n.
\]
It follows that
\[
\dim \rmH^{2b_n+c_n}(\frg,  \tilde K; \bbC)=1.
\]

Put
\be\label{kappa}
\kappa:=\frac{w+w^{\prime}}{2}.
\ee

Our main result is as follows, which can be viewed as the nonvanishing hypothesis for $\GL_n \times \GL_n$ at the critical place $j=-\kappa+\frac{1}{2}$.

\medskip
\noindent \textbf{Theorem A.}\quad
\emph{Assume that
\be\label{assumption}
\kappa \mbox{ is an half integer and fix } j=-\kappa+\frac{1}{2}.
\ee
By restriction of cohomology, the $G$-equivariant linear functional}
\be
\phi_\pi \otimes \phi_F: \pi_\xi \otimes F_\xi^\vee \to \bbC=\bbC \otimes \bbC
\ee
\emph{induces a nonzero map}
\be\label{map-ThmA}
\rmH^{2b_n+c_n}(\frg^3, \tilde K^3; \pi_\xi \otimes F_\xi^{\vee}) \to \rmH^{2b_n+c_n}(\frg, \tilde K; \bbC).
\ee
\medskip

Put $j_0=-\kappa+\frac{1}{2}$. We remark that $I_{j_0}$ is unitary when restricted to $SL(n, \K)$.
For a general critical place $j$ of $\pi_{\mu}\times \pi_{\nu}$, we succeeded in using the translation functor (see Charper 7 of \cite{KV}) to realize $I_j$ as a submodule of $I_{j_0}\otimes F_j$ with multiplicity one, where $F_j$ is certain finite dimensional representation of $\frg$. However, when $j\neq j_0$, the lowest $K$-type of $I_j$ has multiplicity greater than one in
$I_{j_0}\otimes F_j$. This prevents us from obtaining an analogue of Proposition 3.5 of \cite{Sun}.

The outline of the paper is as follows: We recall the definition of critical places and calculated them for $\pi_{\mu}\times \pi_{\nu}$ in Section 2. Then we recall some known results and describe compatibility in a clean fashion in Section 3. We collect all the necessary analysis of finite dimensional representations in Section 4, while those for infinite dimensional representations are presented in Section 5. Then after a short discussion of relative Lie algebra cohomology spaces in Section 6, we prove Theorem A in Section 7.

Finally, we remark that although we have followed  Sun's pioneering paper \cite{Sun} closely for the general approach, there are still many delicate analysis to carry out in our case.

\section{The critical places}\label{sec-cri.pl.}

Assume that $\mu$ is pure with weight $w$, see \eqref{mu-pure-w}. We write $\mu=(\mu^L; \, \mu^R)$, where
\be\label{mu}
\mu^L=(\mu_1, \dots, \mu_n), \quad \mu^R=(\mu_{n+1}, \dots, \mu_{2n})=(w-\mu_n, \dots, w-\mu_1).
\ee
As in Section 2.4 of \cite{Rag}, we put
\be\label{ai-bi}
a_i=\mu_i + \frac{n+1-2i}{2}, \quad b_i=w-a_i, \quad 1\leq i \leq n.
\ee
Now define the representation $J_{\mu}$ to be induced from the Bore subgroup $B_n(\bbC)$ of upper triangular matrices as:
\be\label{J-mu}
J_{\mu}:= \mathrm{Ind}_{B_n(\bbC)}^{\GL_n(\bbC)}\left(z_{11}^{a_1}\overline{z}_{11}^{b_1}\otimes \cdots \otimes z_{nn}^{a_n}\overline{z}_{nn}^{b_n}\right),
\ee
where for any half-integers $a, b$, $z^{a}\overline{z}^{b}$ stands for the character of $\bbC^{\times}$ sending $z$ to $z^{a}\overline{z}^{b}$. It is known that $J_{\mu}$ is the unique representation in $\Omega(\mu)$, see Proposition 2.14 of \cite{Rag}. Thus we can take $\pi_{\mu}$ to be $J_{\mu}$.

Similarly, assume that $\nu$ is pure with weight $w^{\prime}$ and write $\nu=(\nu^L; \, \nu^R)$, where
\be\label{nu}
\nu^L=(\nu_1, \dots, \nu_n), \quad \nu^R=(\nu_{n+1}, \dots, \nu_{2n})=(w^{\prime}-\nu_n, \dots, w^{\prime}-\nu_1).
\ee
Put
\be\label{ci-di}
c_i=\nu_i + \frac{n+1-2i}{2}, \quad d_i=w^{\prime}-c_i, \quad 1\leq i \leq n.
\ee
Now we can take the representation $\pi_{\nu}$ to be the following representation:
\be\label{J-nu}
J_{\nu}:= \mathrm{Ind}_{B_n(\bbC)}^{\GL_n(\bbC)}\left(z_{11}^{c_1}\overline{z}_{11}^{d_1}\otimes \cdots \otimes z_{nn}^{c_n}\overline{z}_{nn}^{d_n}\right).
\ee

\begin{defi}
 An integer $s_0$ is called a \emph{critical place} for $\pi_\mu\times\pi_\nu$ if neither $L(\pi_\mu\times\pi_\nu;s)$ nor $L(\pi_\mu^\vee\times\pi_\nu^\vee;1-s)$ has a pole at $s_0$.
\end{defi}
We are going to determine the set of critical places for $\pi_\mu\times\pi_\nu$.
Define
\be\label{cmunu}
c_{\mu, \nu}=\min_{1\leq i, j\leq n}|\mu_i+\nu_j-\kappa+(n+1)-(i+j)|.
\ee

\begin{prop}\label{prop-critical-places}
The integer $s$ is a critical place for $\pi_{\mu}\times \pi_{\nu}$ if and only if
\be\label{critical-places}
1-\kappa-c_{\mu, \nu} \leq s \leq -\kappa+c_{\mu, \nu}.
\ee
\end{prop}
\begin{proof}
By \eqref{ai-bi}, \eqref{ci-di} and Section 4 of \cite{Kn}, we have
\begin{eqnarray*}
L(\pi_\mu\otimes \pi_\nu; s)
 &\sim& \prod_{i=1}^{n} \prod_{j=1}^{n}  \Gamma\left(s + \frac{a_i+b_i+c_i+d_i}{2}+\frac{|a_i-b_i+c_i-d_i|}{2}\right)\\
&=& \prod_{i=1}^{n} \prod_{j=1}^{n} \Gamma\left(s + \kappa +\frac{|a_i-b_i+c_i-d_i|}{2}\right)\\
&=& \prod_{i=1}^{n} \prod_{j=1}^{n} \Gamma\left(s + \kappa + |\mu_i +\nu_j-\kappa + (n+1)-(i+j)|\right),
\end{eqnarray*}
where ``$\sim$" means equality up to a nonzero real number.
Similarly, we have
\begin{eqnarray*}
L\left(\pi_\mu^{\vee} \otimes \pi_\nu^{\vee}; 1-s\right)
 \sim \prod_{i=1}^{n} \prod_{j=1}^{n} \Gamma\left(1-s - \kappa + |\mu_i +\nu_j-\kappa + (n+1)-(i+j)|\right).
\end{eqnarray*}
Since the gamma function has simple poles at non-positive integers, we have that the integer $s$ is a critical place for $\pi_{\mu}\times \pi_{\nu}$ if and only if
$$
s + \kappa + |\mu_i +\nu_j-\kappa + (n+1)-(i+j)|\geq 1,
$$
and
$$
1-s - \kappa + |\mu_i +\nu_j-\kappa + (n+1)-(i+j)|\geq 1
$$
for all $1\leq i, j \leq n$. Then one arrives at \eqref{critical-places} immediately.
\end{proof}

\section{Compatibility}\label{}

\subsection{}

Let $G=\GL(n, \K)$. Then $G_{\bbC}=\GL(n, \bbC)\times\GL(n, \bbC)$, $\frg_0=\frg\frl(n, \bbC)$ and $\frg =\frg\frl(n, \bbC)\oplus \frg\frl(n, \bbC)$.
Take $T$ be the the Cartan subgroup consisting of the diagonal matrices.
Let $P$ be the standard maximal parabolic subgroup of type $(n-1, 1)$.
Then the Levi factor $L$ of $P$ has the form $A\times M$, where $A$ consists of diagonal matrices with diagonal entries $a_1, \dots, a_1, a_2$; while $M$ consists of block-diagonal matrices with size $(n-1, 1)$. We fix a set of positive roots $\Delta^+(\frg, \frt)$ for  $\Delta(\frg, \frt)$,  and choose compatible positive root systems for $P$ and $L$.
Let $W=W(\frg_, \frt)$ be the Weyl group of $\frg$ with respect to $\frt$,
and similarly $W_L=W(\frl, \frt)$.

Recall that for $j\in\bbZ$, we have the representation
$I_j=\Ind_P^G(H_{j-\frac{1}{2}})$,
where $H_s(p)=|\mspace{-3mu}\det p|_{\K}^s \cdot  |p_{nn}|_{\K}^{-ns}\cdot \eta(p_{nn})^{-1}$.
Here $\eta (z)=\omega_{\pi_{\mu}}(z)\omega_{\pi_{\nu}}(z)$ is the product of the central characters of $\pi_{\mu}$ and $\pi_{\nu}$. Put
\be\label{k-eta}
k_{\eta}:=\sum_{i=1}^{n}(a_i+ c_i)=\sum_{i=1}^{n}(\mu_i+ \nu_i).
\ee
By \eqref{J-mu} and \eqref{J-nu}, we have that
$$
\eta(z)=z^{\sum_{i=1}^{n}(a_i+ c_i)}\, \overline{z}^{\sum_{i=1}^{n}(b_i+ d_i)}
=z^{k_{\eta}} \, \overline{z}^{2n\kappa-k_{\eta}}.
$$

For any $k\in\bbZ$, we denote by $\mathrm{det}^k$ the representation of $GL(n, \bbC)$ with highest weight $(k, k, \dots, k)$.
For $a\in\bbN$, we denote by $\mathrm{Sym}^a$ the representation of $GL(n, \bbC)$ with highest weight $(a, 0, \dots, 0)$; while if $a<0$, we  denote by $\mathrm{Sym}^a$ the representation of $GL(n, \bbC)$ with highest weight $(0, 0, \dots, a)$.

Let us record Proposition 2 and Corollary 1 of \cite{Gre} for later usage.

\begin{prop}\label{prop-Gre} \emph{(Greni\'e)}
Let $\sigma$ and $\tau$ be two finite dimensional irreducible representations of $GL(n, \bbC)$ with highest weights $\lambda=(\lambda_1, \dots, \lambda_n)$ and $\mu=(\mu_1, \dots, \mu_n)$, respectively. Then $\mathrm{det}^d$ occurs in $\tau \otimes \sigma \otimes\mathrm{Sym}^{a}$ $(a\in\bbN)$ if and only if $\max_{1\leq i\leq n}\{\lambda_i+\mu_{n+1-i}\}\leq d\leq \min_{1\leq i\leq n-1}\{\lambda_i+\mu_{n-i}\}$ and $a=nd-\sum_{i=1}^{n}(\mu_i+\lambda_i)\geq 0$.
Moreover, in such a case, $\mathrm{det}^d$ occurs with multiplicity one in $\tau \otimes \sigma\otimes\mathrm{Sym}^{a}$.
\end{prop}

The above proposition is deduced from the Pieri's rule, see Corollary 9.2.4 of \cite{GW}.
The following result can be deduced as Proposition 5 of \cite{Gre}, where the main tool is Theorem III.3.3 of \cite{BW}.

\begin{prop}\label{prop-BW} \emph{(Greni\'e)}
The relative Lie algebra cohomology $\rmH^*(\frg, K; I_j\otimes V_j)$ is non-vanishing if and only if one of the following happens:
\begin{itemize}
\item[(a)] $j\geq \max\left\{ 1-\frac{k_{\eta}}{n}, 1+\frac{k_{\eta}}{n}-2\kappa\right\}$, and $$V_j = (\mathrm{Sym}^{nj + k_{\eta}-n}\otimes \mathrm{det}^{1-j};
\, \mathrm{Sym}^{nj -n - k_{\eta}+ 2n\kappa}\otimes \mathrm{det}^{1-j}).$$
 We then take $l(j)=2(n-1)$.
\item[(b)] $1-2\kappa +\frac{k_{\eta}}{n}\leq j\leq  -\frac{k_{\eta}}{n}$, and $$V_j =
(\mathrm{det}^{-j}\otimes \mathrm{Sym}^{nj + k_{\eta}};
\, \mathrm{Sym}^{nj -n - k_{\eta}+ 2n\kappa}\otimes \mathrm{det}^{1-j}).$$
 We then take $l(j)=n-1$.
 \item[(c)] $1 -\frac{k_{\eta}}{n}\leq j\leq  -2\kappa +\frac{k_{\eta}}{n}$, and $$V_j =
(\mathrm{Sym}^{nj + k_{\eta}-n}\otimes \mathrm{det}^{1-j};
\, \mathrm{det}^{-j} \otimes \mathrm{Sym}^{nj + 2n\kappa - k_{\eta}}).$$
 We then take $l(j)=n-1$.
  \item[(d)] $ j\leq \min\left\{-\frac{k_{\eta}}{n}, \frac{k_{\eta}}{n} -2\kappa\right\}$, and $$V_j =
(\mathrm{det}^{-j}\otimes \mathrm{Sym}^{nj + k_{\eta}};
\, \mathrm{det}^{-j} \otimes \mathrm{Sym}^{nj + 2n\kappa - k_{\eta}}).$$
 We then take $l(j)=0$.
 \end{itemize}
The cohomology is equal to $\bbC$ in degree $l(j)$, $\bbC^2$ in degree $l(j)+1$, $\bbC$ in degree $l(j)+2$ and zero in all other degrees.
\end{prop}

We put
\be\label{M-m}
M_{\mu, \nu}^n := \max_{i+j=n\atop 1\leq i, j\leq n} \left\{\mu_i+\nu_{j}\right\}, \qquad
m_{\mu, \nu}^n := \min_{i+j=n\atop 1\leq i, j\leq n} \left\{\mu_i+\nu_{j}\right\}.
\ee
The numbers $M_{\mu, \nu}^{n+1}$, $m_{\mu, \nu}^{n+1}$, $M_{\mu, \nu}^{n+2}$, $m_{\mu, \nu}^{n+2}$ are interpreted similarly.

\begin{prop}\label{prop-compatible} \emph{(i)} In the setting of Proposition \ref{prop-BW}(b), we have
 $$\mathrm{dim}\, \mathrm{Hom}_{G_{\bbC}}(F_{\mu}^{\vee}\otimes F_{\nu}^{\vee}\otimes V_j, \bbC)\leq 1.$$
 Moreover, equality holds if and only if
\be\label{cond-b}
- m_{\mu, \nu}^{n} \leq  j\leq -M_{\mu, \nu}^{n+1}
\quad\mbox{and}\quad
1-2\kappa + M_{\mu, \nu}^{n+1} \leq  j\leq 1-2\kappa +m_{\mu, \nu}^{n}.
\ee

\noindent\emph{(ii)} In the setting of Proposition \ref{prop-BW}(c), we have
 $$\mathrm{dim}\,\mathrm{Hom}_{G_{\bbC}}(F_{\mu}^{\vee}\otimes F_{\nu}^{\vee}\otimes V_j, \bbC)\leq 1.$$ Moreover, equality holds if and only if
\be\label{cond-c}
 1-m_{\mu, \nu}^{n+1} \leq  j\leq 1-M_{\mu, \nu}^{n+2}
\quad\mathrm{and}\quad
M_{\mu, \nu}^{n+2} -2\kappa \leq  j\leq -2\kappa + m_{\mu, \nu}^{n+1}.
\ee
\end{prop}
\bp
For (i), note that
\begin{eqnarray*}
\mathrm{Hom}_{G_{\bbC}}(F_{\mu}^{\vee}\otimes F_{\nu}^{\vee}\otimes V_j, \bbC) &=& \mathrm{Hom}_{G_{\bbC}}(V_j, F_{\mu}\otimes F_{\nu})\\
&=& \mathrm{Hom}_{G}(\det^{-j}  \otimes \mathrm{Sym}^{nj+k_{\eta}}, E_{\mu^L}\otimes E_{\nu^L})\\
&\times&  \mathrm{Hom}_{G}(\mathrm{Sym}^{nj-n+2n\kappa-k_{\eta}}\otimes \det^{1-j}, E_{\mu^R}\otimes E_{\nu^R})\\
&=& \mathrm{Hom}_{G}(\det^{-j}, E_{\mu^L}\otimes E_{\nu^L} \otimes  \mathrm{Sym}^{-nj-k_{\eta}})\\
&\times&  \mathrm{Hom}_{G}(E_{\mu^R}^{\vee}\otimes E_{\nu^R}^{\vee}\otimes \mathrm{Sym}^{nj-n+2n\kappa-k_{\eta}}, \det^{j-1}).
\end{eqnarray*}
Here $E_{\lambda}$ denotes the representation of $GL(n, \bbC)$ with highest weight $\lambda$.
By Proposition \ref{prop-Gre}, one has that  $\mathrm{Hom}_{G}(\det^{-j}, E_{\mu^L}\otimes E_{\nu^L} \otimes  \mathrm{Sym}^{-nj-k_{\eta}})$ is non-vanishing if and only if
$$
\max_{1\leq i\leq n}\{\mu_i + \nu_{n+1-i}\}\leq -j \leq \min_{1\leq i\leq n-1}\{\mu_i + \nu_{n-i}\};
$$
and  that $\mathrm{Hom}_{G}(E_{\mu^R}^{\vee}\otimes E_{\nu^R}^{\vee}\otimes \mathrm{Sym}^{nj-n+2n\kappa-k_{\eta}}, \det^{j-1})$ is non-vanishing if and only if
$$
\max_{1\leq i\leq n}\{\mu_i + \nu_{n+1-i}\}-2\kappa\leq j-1 \leq \min_{1\leq i\leq n-1}\{\mu_i + \nu_{n-i}\}-2\kappa.
$$
Thus one arrives at \eqref{cond-b}, as desired.

The proof for (ii) is similar, we omit it here.
\ep

\begin{lemma}\label{lemma-compare} \emph{(i)} Suppose that \eqref{cond-b} has solutions, then they coincide with those to \eqref{critical-places}.

\noindent\emph{(ii)} Suppose that \eqref{cond-c} has solutions, then they coincide with those to \eqref{critical-places}.
\end{lemma}
\bp
It is easy to see that \eqref{cond-b} has solutions if and only if
$$
M_{\mu, \nu}^{n+1}\leq m_{\mu, \nu}^{n}, \quad 1-2\kappa +M_{\mu, \nu}^{n+1}\leq -M_{\mu, \nu}^{n+1}, \quad -m_{\mu, \nu}^{n}\leq 1-2\kappa + m_{\mu, \nu}^{n},
$$
if and only if
\be\label{cond-b-simplified}
M_{\mu, \nu}^{n+1}\leq \kappa -\frac{1}{2}\leq m_{\mu, \nu}^{n}.
\ee
Now suppose that \eqref{cond-b-simplified} holds. On one hand, for any $i, j$ such that $i+j\leq n$, we have
$$
\mu_i+\nu_j-\kappa+(n+1)-(i+j)\geq m_{\mu, \nu}^{n}-\kappa +1\geq \frac{1}{2} > 0.
$$
On the other hand, for any $i, j$ such that $i+j\geq n+1$, we have
$$
\mu_i+\nu_j-\kappa+(n+1)-(i+j)\leq M_{\mu, \nu}^{n+1}-\kappa\leq -\frac{1}{2} < 0 .
$$
We conclude that
$$
c_{\mu, \nu}=
\min\left\{
m_{\mu, \nu}^{n}-\kappa +1, \kappa - M_{\mu, \nu}^{n+1}
\right\}.
$$
Then one sees that \eqref{critical-places} is just a reformulation of \eqref{cond-b}. This proves (i).

For (ii), it is easy to see that \eqref{cond-c} has solutions
if and only if
\be\label{cond-c-simplified}
M_{\mu, \nu}^{n+2}\leq \kappa + \frac{1}{2}\leq m_{\mu, \nu}^{n+1}.
\ee
The remaining discussion is similar to the previous case, we omit the details.
\ep
\begin{rmk}\label{rmk-lemma-compare}
Note that if \eqref{cond-b-simplified} holds, we have
$$
\frac{k_{\eta}}{n}=\frac{1}{n}\sum_{i=1}^{n}(\mu_i+\nu_{n+1-i})\leq  M_{\mu, \nu}^{n+1}\leq \kappa - \frac{1}{2}.
$$
Then one sees that ``$1-2\kappa +\frac{k_{\eta}}{n}\leq j\leq  -\frac{k_{\eta}}{n}$" holds since \eqref{cond-b-simplified} is equivalent to \eqref{cond-b}. Similarly, if \eqref{cond-c-simplified} holds, we have
$$
\frac{k_{\eta}}{n}=\frac{1}{n}\sum_{i=1}^{n}(\mu_i+\nu_{n+1-i})\geq  m_{\mu, \nu}^{n+1}\geq \kappa + \frac{1}{2}.
$$
Then one sees that ``$1-\frac{k_{\eta}}{n}\leq j\leq  -2\kappa + \frac{k_{\eta}}{n}$" holds since \eqref{cond-c-simplified} is equivalent to \eqref{cond-c}.
\end{rmk}

\begin{defi}\label{def-compatible} We say that two pure weights $\mu$ and $\nu$ are \emph{compatible} if \eqref{cond-b} or \eqref{cond-c} has solutions.
\end{defi}

By Lemma \ref{lemma-compare} and Remark \ref{rmk-lemma-compare}, we see that if $\mu$ and $\nu$ are compatible,  the solutions to \eqref{cond-b} or \eqref{cond-c} coincide with the critical places for $\pi_{\mu}\times \pi_{\nu}$.

\section{Finite dimensional representations}

\subsection{Cartan component and PRV component}
In this subsection, we let $\frg$ be a finite dimensional simple Lie algebra over $\bbC$. Let $\frh$ be a Cartan subalgebra of $\frg$. Fix a positive root system $\Delta^+(\frg, \frh)$. Then for any dominant integral weight $\lambda\in\frh^*$, we denote by $F_{\lambda}$ the finite dimensional irreducible representation of $\frg$ with highest weight $\lambda$. Let $W=W(\frg, \frh)$ be the Weyl group. Take $w_0$ as the longest element of $W$. Then $F_{\lambda}$ has lowest weight $w_0\lambda$.
Let $F_{\mu}$ be another irreducible representation of $\frg$ with highest weight $\lambda$.  Then it is well-known that $F_{\lambda+\mu}$ occurs with multiplicity one in $F_{\lambda}\otimes F_{\mu}$. We call $F_{\lambda+\mu}$ the \emph{Cartan component} of $F_{\lambda}\otimes F_{\mu}$. On the other hand, let $\sigma$ be the irreducible representation of $\frg$ with extremal weight $\lambda + w_0 \mu$. Then by Corollary 1 to Theorem 2 of \cite{PRV}, we have that $\sigma$ occurs with  multiplicity one in $F_{\lambda}\otimes F_{\mu}$. We call $\sigma$ the \emph{PRV component} of $F_{\lambda}\otimes F_{\mu}$. Here, ``PRV" stands for the initials of Parthasarathy, Rao, and Varadarajan.

\subsection{Minimal $K$-types}
Recall that $\K$ is a topological field  isomorphic to $\C$.
We put $G=GL(n, \K)$, $K=U(n)$ and $\tilde{K}=GU(n)$. Recall from \eqref{eq: comlexify} that $G_{\bbC}\cong GL(n, \bbC)\times GL(n, \bbC)$, and that $\frg=\frg\frl(n, \bbC)\times \frg\frl(n, \bbC)$ is the complexified Lie algebra of $G$.
Fix a Borel subgroup $B$ of $G$ consisting of the upper triangular matrices, and choose a positive root system for $G$ accordingly. We use
$E_{\lambda}$ to denote the irreducible $G$ representation with highest weight $\lambda$.
Weyl's unitary trick (see Proposition 7.15 of \cite{Kn02}) allows us to shift freely between $GL(n, \K)$ representations and $U(n)$ representations. Thus we shall abuse notation slightly, and denote by $E_{\lambda}$ the $K$-type  with highest weight $\lambda$ as well.

From now on, we always assume that $\mu$ and $\nu$ are compatible, and that $j\in\bbZ$ is a critical place for $\pi_{\mu}\times \pi_{\nu}$. Let us revisit the representation $I_j$ of $G$, which is defined in \eqref{Ij}. The following lemma determines the $K$-types of $I_j$.

\begin{lemma}\label{lemma-Ij-min-K} The $K$-types in $I_j$ are exactly those with highest weights $(m, 0, \dots, 0, 2n\kappa-2 k_{\eta}-m)$, where the integer $m\geq \max\{0, 2n\kappa-2 k_{\eta}\}$. Moreover, $I_j|_K$ is multiplicity free.
\end{lemma}
\bp
Recall that $H_j(p)=|\mspace{-3mu}\det p|_{\K}^j \cdot  |p_{nn}|_{\K}^{-nj}\cdot \eta(p_{nn})^{-1}$, where $\eta(z)^{-1}=z^{-k_{\eta}}\overline{z}^{k_{\eta}-2n\kappa}$.
Let $V$ be any $K$-type. By Frobenius reciprocity, we have
\be\label{Frobenius-Ij}
\mathrm{Hom}_K (I_j, V)\cong \mathrm{Hom}_{P\cap K}(H_j|_{P\cap K}, V|_{P\cap K}).
\ee
Note that $P\cap K=U(n-1)\times U(1)$, and that $H_j|_{P\cap K}$ has weight $(0, \dots, 0, 2n\kappa-2k_{\eta})$. Suppose that $V$ has highest weight $(\lambda_1, \dots, \lambda_n)$. Then $V|_{P\cap K}$ contains $H_j|_{P\cap K}$ if and only if
$$
\lambda_1\geq 0\geq \lambda_2 \geq 0 \geq \cdots \geq 0\geq \lambda_{n-1}\geq 0\geq \lambda_n,
$$
and that $$\sum_{i=1}^{n}\lambda_i=2n\kappa-2k_{\eta}.$$
Therefore, $$\lambda_1=m, \lambda_2=\cdots=\lambda_{n-1}=0, \lambda_n=2n\kappa-2k_{\eta}-m,$$
where $m\geq \max\{0, 2n\kappa-2k_{\eta}\}$.
\ep

Denote the minimal $K$-type  of $I_j$ by $\sigma_j^+$. In case (b) of Proposition \ref{prop-BW}, we have $2n\kappa-2k_{\eta}\geq n$. Thus by Lemma \ref{lemma-Ij-min-K},
\be\label{sigma-j-b}
\sigma_j^+=E_{(2n\kappa-2k_{\eta}, 0, \cdots, 0)}.
\ee
In case (c) of Proposition \ref{prop-BW}, we have $2n\kappa-2k_{\eta}\leq -n$. Thus by Lemma \ref{lemma-Ij-min-K},
\be\label{sigma-j-c}
\sigma_j^+=E_{(0, \cdots, 0, 2n\kappa-2k_{\eta})}.
\ee

\begin{lemma}\label{lemma-Jmu-min-K} The minimal $K$-type, denoted by $\tau_{\mu}^{+}$, of $J_{\mu}$ is the one with highest weight $(2\mu_1-w+n-1, \dots, 2\mu_n-w-n+1)$. Moreover, $\tau_{\mu}^{+}$ occurs with multiplicity one in $J_{\mu}$.
\end{lemma}

This lemma is also deduced from the Frobenius reciprocity, we omit the details.

\subsection{Analysis of the multiplicity of certain $GL(n, \K)$ representations}
 Let $\tilde{\frp}$ be the Lie algebra of traceless matrices $M\in\frg$ such that $\overline{M}^{t}=M$. Then we have that $$\frg/\tilde{\frk}\cong \tilde{\frp}$$ as $K$ representations.

Let $V$ be the standard representation of $G$, and let $V^{\vee}$ be its contragredient. Let $e_1, \dots, e_n$ be the standard basis of $V$, and let $e_1^*, \dots, e_n^*$ be the corresponding dual dual basis. Then it is easy to check that the vector $e_{ij}:=e_i\otimes e_j^*$ has weight $\epsilon_i-\epsilon_j$. Here $\epsilon_i=(0, \dots, 1, \dots, 0)$, where the unique 1 is the $i$-th entry. Now as proved in Proposition 6 of \cite{Gre},
\be\label{V-Vcheck}
V\otimes V^{\vee}\cong \bbC\oplus \tilde{\frp}
\ee
and as $K$ representations,
\be\label{p-tilde}
\tilde{\frp}\cong E_{(1, 0, \dots, 0, -1)}.
\ee
Indeed, the vector $\sum_{i=1}^{n} e_{ii}$ generates the trivial representation $\bbC$. On the other hand, one can check that  $e_{1n}$ is a highest weight vector, and that the $\dim E_{(1, 0, \dots, 0, -1)} = n^2-1$. This verifies \eqref{V-Vcheck} as well as \eqref{p-tilde}. Let us fix a basis consisting of weight vectors of $\tilde{\frp}$:
\be\label{basis-p-tilde}
e_{ij}, \, 1\leq i\neq j \leq n;\, e_{11}-e_{kk}, \, 2\leq k\leq n.
\ee

\begin{lemma}\label{lemma-mult-one-sigma-n-wedge}
The $GL(n, \K)$ representation $E_{(n-1, -1, \dots, -1)}$ and its contragredient both occur with multiplicity one in $\wedge^{n-1}(\frg/\tilde{\frk})$.
\end{lemma}
\bp
Firstly, it is direct to check that
\be\label{vector-n-1}
e_{12}\wedge \cdots \wedge e_{1n}\in\wedge^{n-1}\tilde{\frp}
\ee
is a highest weight vector with weight $(n-1, -1, \dots, -1)$.
Thus $E_{(n-1, -1, \dots, -1)}$ occurs  in $\wedge^{n-1}\tilde{\frp}\cong \wedge^{n-1}(\frg/\tilde{\frk})$.

Now suppose that $v_0$ is the highest weight vector of an occurrence of $E_{(n-1, -1, \dots, -1)}$ in $\wedge^{n-1}\tilde{\frp}$. Take an arbitrary linear component of $v_0$, we may and we will assume that it has the form $A_1\wedge \cdots \wedge A_{n-1}$, where each $A_i$ is from \eqref{basis-p-tilde}.  Observe that there is no repetitions among these $A_i$, and all of them are weight vectors. We collect the corresponding \emph{non-zero} weights as  $\alpha_1, \dots, \alpha_r$, which must be pairwise distinct. We have $r\leq n-1$, and
$$
\alpha_1 + \dots + \alpha_r=(n-1, -1, \dots, -1).
$$
One sees that this equation has solution if and only if $r=n-1$ and
$$\{\alpha_1, \dots, \alpha_{n-1}\}=\{\epsilon_1-\epsilon_2, \dots, \epsilon_1-\epsilon_n\}.$$
Therefore, $A_1\wedge \cdots \wedge A_{n-1}$ is a scalar multiple of \eqref{vector-n-1}. Then so is $v_0$. Thus $E_{(n-1, -1, \dots, -1)}$ occurs with multiplicity one in $\wedge^{n-1}\tilde{\frp}$.

Since $\tilde{\frp}$ is self-dual as a $G$ representation, we conclude immediately that its contragredient representation $E_{(1, \dots, 1, 1-n)}$ occurs exactly once in
$\wedge^{n-1}\tilde{\frp}$ as well.
\ep
\begin{rmk}\label{rmk-lemma-mult-one-sigma-n}
It is not hard to prove that a vector of the following form occurs in the unique realization of
$E_{(n-1, 1, \dots,  1)}$ in $\wedge^{n-1}(\frg/\tilde{\frk})$:
\be\label{vector-sigma-n}
(e_{11}-e_{22})\wedge \cdots \wedge (e_{11}-e_{nn}) + \mbox{other terms},
\ee
where each linear component in ``other terms" is different from the leading one.
\end{rmk}

 Recall that the $G_{\bbC}$ representation $F_{\mu}$ has highest weight $(\mu_L, \mu_R)$. Define the $K$-type
\be\label{tau-mu}
\tau_{\mu} := E_{\mu_L - w_0\mu_R}=E_{(2\mu_1-w, \dots, 2\mu_n-w)},
\ee
where $w_0$ is the longest element of $S_n$.
We define $\tau_{\nu}$ similarly. Recall that the minimal $K$-type of $J_{\mu}$ is $\tau_{\mu}^{+}$. We define $\tau_n$ to be the PRV-component of  $\tau_{\mu}^{\vee}\otimes\tau_{\mu}^{+}$. By Lemma \ref{lemma-Jmu-min-K} and \eqref{tau-mu}, we have
\be\label{tau-n}
\tau_n:=E_{(n-1, n-3, \dots, 1-n)}.
\ee
In particular, it is independent of $\mu$. Thus PRV-component of  $\tau_{\nu}^{\vee}\otimes\tau_{\nu}^{+}$ is $\tau_n$ as well.

\begin{lemma}\label{lemma-mult-one-tau-n}
The $K$-type $\tau_n$  occurs with multiplicity one in $\wedge^{b_n}(\frg/\tilde{\frk})$.
\end{lemma}
\bp
Recall that $b_n=\frac{n(n-1)}{2}$.
Firstly, it is direct to check that
\be\label{vector-bn}
e_{12}\wedge \cdots \wedge e_{1n}\wedge e_{23}\wedge \cdots\wedge e_{2n}\wedge \cdots \wedge e_{n-1, n} \in\wedge^{b_{n}}\tilde{\frp}
\ee
is a highest weight vector with weight $(n-1, n-3, \dots, 3-n, 1-n)$.
Thus $\tau_n$ occurs  in $\wedge^{b_n}\tilde{\frp}\cong \wedge^{b_n}(\frg/\tilde{\frk})$.

Now suppose that $v_0$ is the highest weight vector of an occurrence of $\tau_n$ in $\wedge^{b_n}\tilde{\frp}$. Take an arbitrary linear component of $v_0$, we may and we will assume that it has the form $A_1\wedge \cdots \wedge A_{b_n}$, where each $A_i$ is from \eqref{basis-p-tilde}.  Observe that there is no repetitions among those $A_i$, and all of them are weight vectors. We collect the corresponding \emph{non-zero} weights as  $\alpha_1, \dots, \alpha_r$, which must be pairwise distinct. We have $r\leq b_n$, and
$$
\alpha_1 + \dots + \alpha_r=(n-1, n-3, \dots, 3-n, 1-n).
$$
Note that the RHS is the sum of all the positive roots. Thus this equation has solution if and only if $r=b_n$ and
$$\{\alpha_1, \dots, \alpha_{b_n}\}=\{\epsilon_i-\epsilon_j\mid 1\leq i<j \leq n\}.$$
Therefore, $A_1\wedge \cdots \wedge A_{b_n}$ is a scalar multiple of \eqref{vector-bn}. Then so is $v_0$. We conclude that $\tau_n$ occurs with multiplicity one in $\wedge^{b_n}\tilde{\frp}$.
\ep
\begin{rmk}\label{rmk-lemma-mult-one-tau-n}
One sees easily that the vector $e_{21}\wedge \cdots \wedge e_{n1}\wedge e_{32}\wedge \cdots\wedge e_{n2}\wedge \cdots \wedge e_{n, n-1}
$ occurs in the unique realization of
$\tau_n$ in $\wedge^{b_n}(\frg/\tilde{\frk})$:

\end{rmk}

Let us come back to the setting of Proposition \ref{prop-BW}.
Recall that $\sigma_j^+$ is the minimal $K$-type of $I_j$. Denote the highest weight of the $G_{\bbC}$ representation $V_j$ by $(\lambda_L, \lambda_R)$, and define the $K$-type
\be\label{sigma-j}
\sigma_j := E_{\lambda_L - w_0\lambda_R},
\ee
Let $\sigma_n$ be the PRV component of $\sigma_j\otimes \sigma_j^+$.

\begin{lemma}\label{lemma-mult-one-sigma-n} In Proposition \ref{prop-BW}  (b) and (c),
the $K$-type $\sigma_n$ occurs with multiplicity one in $\wedge^{n-1}(\frg/\tilde{\frk})$.
\end{lemma}
\bp
In case (b) of Proposition \ref{prop-BW},
$\sigma_j^+$ is given by \eqref{sigma-j-b}. On the other hand,
$\sigma_j$ has highest weight $(-1, \cdots, -1, 2k_{\eta}-2n\kappa+n-1)$.
Thus the PRV component of $\sigma_j\otimes \sigma_j^+$ is $E_{(n-1, -1, \dots, -1)}$.

In case (c) of Proposition \ref{prop-BW}, $\sigma_j^+$ is given by \eqref{sigma-j-c}.  On the other hand, $\sigma_j$ has highest weight
$(2k_{\eta}-2n\kappa- n+1, 1, \cdots, 1, 1)$.
Thus the PRV component  of $\sigma_j\otimes \sigma_j^+$ is $E_{(1, \dots, 1, 1-n)}$.

Now in each case, the desired conclusion follows from Lemma \ref{lemma-mult-one-sigma-n-wedge}.
\ep

\subsection{Analysis of $\tau_n\otimes\tau_n\otimes\sigma_n$}
Let $\tau_n\boxtimes\tau_n\boxtimes\sigma_n$ be the outer tensor product of $\tau_n$, $\tau_n$ and $\sigma_n$. That is, it has the same underlying space as $\tau_n\otimes\tau_n\otimes\sigma_n$, but it is a representation of $K\times K\times K$.
When restricted to $K$ via the embedding $k\mapsto (k, k, k)$, it becomes the usual tensor product $\tau_n\otimes\tau_n\otimes\sigma_n$. Note that
$$
\frg/\tilde{\frk}\oplus \frg/\tilde{\frk}\oplus \frg/\tilde{\frk}
\cong \tilde{\frp}\oplus \tilde{\frp}\oplus \tilde{\frp}
$$
as $K\times K\times K$ representations. Given a vector $A\in \tilde{\frp}$, we denote by $A^1=(A, 0, 0)$, $A^2=(0, A, 0)$ and $A^3=(0, 0, A)$. Therefore,
\be\label{basis-p-tilde-triple}
e_{ij}^l, \, 1\leq i\neq j \leq n;\, e_{11}^l-e_{kk}^l, \, 2\leq k\leq n,
\ee
where $l=1, 2$ or $3$,
is a basis of $\tilde{\frp}\oplus \tilde{\frp}\oplus \tilde{\frp}$ consisting of weight vectors.

\begin{lemma}\label{lemma-mult-outer-tensor}
The representation $\tau_n\boxtimes\tau_n\boxtimes\sigma_n$ occurs with multiplicity one in $\wedge^{2b_n+c_n}(\frg/\tilde{\frk}\oplus \frg/\tilde{\frk}\oplus \frg/\tilde{\frk})$.
\end{lemma}
\bp
We consider the case $\sigma_n=E_{(n-1, -1, \dots, -1)}$. The other case $\sigma_n=E_{( 1, \dots, 1, 1-n)}$ is similar.
Firstly, it is direct to check that
\be\label{vector-triple}
e_{12}^1\wedge  \cdots \wedge e_{n-1, n}^1
\wedge
e_{12}^2\wedge \cdots \wedge e_{n-1, n}^2
\wedge
(e_{11}^3-e_{22}^3)\wedge\cdots\wedge (e_{11}^3-e_{nn}^3)
\in\wedge^{2b_n+c_n}(\tilde{\frp}\oplus \tilde{\frp}\oplus \tilde{\frp})
\ee
is a highest weight vector with weight $$(n-1, \dots, 1-n; n-1, \dots, 1-n; n-1, -1, \dots, -1).$$
Thus $\tau_n\boxtimes\tau_n\boxtimes\sigma_n$ occurs  in $\wedge^{2b_n+c_n}(\tilde{\frp}\oplus \tilde{\frp}\oplus \tilde{\frp})$.

Now suppose that $v_0$ is the highest weight vector of an occurrence of $\tau_n\boxtimes\tau_n\boxtimes\sigma_n$ in $\wedge^{2b_n+c_n}(\tilde{\frp}\oplus \tilde{\frp}\oplus \tilde{\frp})$. Based on the proofs of Lemmas \ref{lemma-mult-one-tau-n} and \ref{lemma-mult-one-sigma-n}, it is now easy to see that $v_0$ is a scalar multiple of \eqref{vector-triple}. Thus $\tau_n\boxtimes\tau_n\boxtimes\sigma_n$ occurs with multiplicity one in $\wedge^{2b_n+c_n}(\tilde{\frp}\oplus \tilde{\frp}\oplus \tilde{\frp})$.
\ep

For simplicity, we put $K^3:=K\times K\times K$.
Using Lemma \ref{lemma-mult-outer-tensor}, we fix a nonzero element
\be
\eta_n\in \mathrm{Hom}_{K^3}(\wedge^{2b_n+c_n}(\frg/\tilde{\frk}\oplus \frg/\tilde{\frk}\oplus \frg/\tilde{\frk}), \tau_n\boxtimes\tau_n\boxtimes\sigma_n).
\ee
Write
$$\iota_n: \wedge^{2b_n+c_n}(\frg/\tilde{\frk}) \to \wedge^{2b_n+c_n}(\frg/\tilde{\frk}\oplus \frg/\tilde{\frk}\oplus \frg/\tilde{\frk})
$$
for the natural embedding.

\begin{lemma}\label{lemma-inv-dim-one-tau-n-tau-n-sigma-n}
We have  $\mathrm{dim}(\tau_n\otimes \tau_n\otimes \sigma_n)^K=1$, where $\sigma_n=E_{(n-1, -1, \dots, -1)}$ or $E_{( 1, \dots, 1, 1-n)}$.
\end{lemma}
This lemma follows directly from Proposition \ref{prop-Gre}, we omit the details.

\begin{lemma}\label{lemma-mult-one-tau-n-tau-n-sigma-n} The composition
\be\label{composition-iota-eta}
\begin{CD}
\wedge^{2b_n+c_n}(\frg/\tilde{\frk}) @> \iota_n > > \wedge^{2b_n+c_n}(\frg/\tilde{\frk}\oplus \frg/\tilde{\frk}\oplus \frg/\tilde{\frk})  @> \eta_n > > \tau_n\boxtimes\tau_n\boxtimes\sigma_n
\end{CD}
\ee
is nonzero. It is image
is equal to $(\tau_n\otimes \tau_n\otimes \sigma_n)^K$.
\end{lemma}
\bp
We consider the case $\sigma_n=E_{(n-1, -1, \dots, -1)}$. The other case $\sigma_n=E_{( 1, \dots, 1, 1-n)}$ is similar.
Let us equip $\tilde{\frp}$ with a $K$-invariant positive definite Hermitian form $\langle\, , \,\rangle$ such that \eqref{basis-p-tilde} is an orthogonal basis. It induces a $K^3$-invariant positive definite Hermitian form $\langle\, , \,\rangle_{\wedge}$ on $\wedge^{2b_n+c_n}(\tilde{\frp}\oplus \tilde{\frp}\oplus \tilde{\frp})$.
Note that
\[\begin{CD}
\wedge^{2b_n+c_n}(\frg/\tilde{\frk}\oplus \frg/\tilde{\frk}\oplus \frg/\tilde{\frk})  @> \eta_n > > \tau_n\boxtimes\tau_n\boxtimes\sigma_n
\end{CD}
\]
is a scalar multiple of the orthogonal projection. By Remarks \ref{rmk-lemma-mult-one-sigma-n}  and \ref{rmk-lemma-mult-one-tau-n}, in the unique realization of $\tau_n\boxtimes\tau_n\boxtimes\sigma_n$ in $\wedge^{2b_n+c_n}(\tilde{\frp}\oplus \tilde{\frp}\oplus \tilde{\frp})$ one can find a vector of the form
$$
w_0 := v_0\wedge v_0^{\prime}\wedge ((e_{11}^3-e_{22}^3)\wedge \cdots \wedge (e_{11}^3-e_{nn}^3) + \mbox{other terms}),$$
where $v_0=e_{12}^1\wedge  \cdots \wedge e_{n-1, n}^1$, $v_0^{\prime}=e_{21}^2\wedge \cdots \wedge e_{n, n-1}^2$. Here each linear component in ``other terms" is different from $(e_{11}^3-e_{22}^3)\wedge \cdots \wedge (e_{11}^3-e_{nn}^3)$, and has the form
$A_1^3\wedge \cdots \wedge A_{n-1}^3$, where each $A_i$ is in \eqref{basis-p-tilde},
see \eqref{vector-sigma-n}. Put
$$
u_0:= e_{12} \wedge \cdots \wedge e_{n-1,n} \wedge e_{21} \wedge\cdots\wedge e_{n, n-1}\wedge
(e_{11}-e_{22})\wedge \cdots \wedge (e_{11}-e_{nn})\in\wedge^{2b_n+c_n}\tilde{\frp}.
$$
Then one sees easily that the ``other terms" has no contribution in $\langle\iota_n(u_0), w_0\rangle_{\wedge}$, and that the latter is nonzero. Thus $\eta_n \circ \iota_n$ is nonzero.

The image of the composition \eqref{composition-iota-eta} is a nonzero subspace of
$(\tau_n\otimes \tau_n\otimes \sigma_n)^K$. By Lemma
\ref{lemma-inv-dim-one-tau-n-tau-n-sigma-n}, the latter space is one dimensional. Therefore the second assertion follows.
\ep

\subsection{Analysis of $F_{\xi}^{\vee}$}
Recall from the introduction that
$$
F_{\xi}^{\vee} := F_{\mu}^{\vee}\otimes F_{\nu}^{\vee}\otimes V_j.
$$
As a $GL_n(\bbC)\times GL_n(\bbC)$ representation, $F_{\mu}^{\vee}\cong E_{\mu_L}^{\vee}\boxtimes E_{\mu_R}^{\vee}$. When restricted to $K$, we have $$F_{\mu}^{\vee}\cong E_{\mu_L}^{\vee}\otimes E_{\mu_R}.$$ Similarly, $F_{\nu}^{\vee}\cong E_{\nu_L}^{\vee}\otimes E_{\nu_R}$, and $V_j\cong E_{\lambda_L}\otimes E_{\lambda_R}^{\vee}$ as representations of $K$.  Recall the $K$-types $\tau_{\mu} = E_{\mu_L-w_0\mu_R}$, $\tau_{\nu} = E_{\nu_L-w_0\nu_R}$ and $\sigma_j = E_{\lambda_L-w_0\lambda_R}$.
Put
\be\label{tau-xi-check}
\tau_{\xi}^{\vee} := \tau_{\mu}^{\vee}\otimes \tau_{\nu}^{\vee}\otimes \sigma_j,
\ee
and
\be\label{tau-xi-plus}
\tau_{\xi}^{+} := \tau_{\mu}^{+}\otimes \tau_{\nu}^{+}\otimes \sigma_j^+.
\ee

\begin{lemma}\label{lemma-inv-dim-one-tau-xi-check}
We have  $\mathrm{dim}(\tau_{\xi}^{\vee})^K=1$.
\end{lemma}
\bp
We only prove it for Proposition \ref{prop-BW} (b), where $\sigma_j=\det^{-1}\otimes \mbox{Sym}^{2k_{\eta}-2n\kappa+n}$. Thus
$$\mathrm{Hom}_{K}(\bbC, \tau_{\xi}^{\vee})\cong \mathrm{Hom}_{K}(\tau_{\mu}\otimes\tau_{\nu}\otimes\mbox{Sym}^{2n\kappa-2k_{\eta}-n},  \det^{-1}).$$
By Proposition \ref{prop-Gre}, the latter space is one dimensional if and only if  $2 M_{\mu, \nu}^{n+1} -2 \kappa \leq -1 \leq 2 m_{\mu, \nu}^n -2\kappa$. This is exactly \eqref{cond-b-simplified}.
\ep

\begin{lemma}\label{lemma-inv-dim-one-tau-xi-plus}
We have  $\mathrm{dim}(\tau_{\xi}^{+})^K=1$.
\end{lemma}
\bp
We only prove it for case (b) of Proposition \ref{prop-BW}. Indeed, by Proposition \ref{prop-Gre}, $\mathrm{Hom}_{K}(\bbC, \tau_{\xi}^{+})$ is one dimensional if and only if  $2 M_{\mu, \nu}^{n+1} -2 \kappa \leq 0\leq 2 m_{\mu, \nu}^n -2\kappa +2$. On one hand, by
\eqref{cond-b-simplified}, we have
$$
2 M_{\mu, \nu}^{n+1} -2 \kappa\leq 2(\kappa-\frac{1}{2}) -2 \kappa\leq -1<0.
$$
On the other hand, by
\eqref{cond-b}, we have
$$
2 m_{\mu, \nu}^n -2\kappa +2 \geq 2(\kappa-\frac{1}{2}) -2\kappa +2 \geq 1>0.
$$
This finishes the proof.
\ep

\begin{lemma}\label{lemma-mult-one-sigma-j}
The $K\times K\times K$ representation $\tau_{\xi}^{\vee}$ occurs with multiplicity one in $F_{\xi}^{\vee}$.
Moreover, every non-zero element of $\mathrm{Hom}_{G_{\bbC}}(F_{\xi}^{\vee}, \bbC)$ does not vanish on $\tau_{\xi}^{\vee}\subset F_{\xi}^{\vee}$.
\end{lemma}
\bp
Note that $\tau_{\mu}^{\vee}$ is the Cartan component of $F_{\mu}^{\vee}$.
Similarly, $\tau_{\nu}^{\vee}$ (resp. $\sigma_j$) is the Cartan component of $F_{\nu}^{\vee}$ (resp. $V_j$). Thus the $K\times K\times K$ representation $\tau_{\xi}^{\vee}$ occurs with multiplicity one in $F_{\xi}^{\vee}$. This proves the first assertion.

We fix a nonzero element $\phi_F$ in $\mathrm{Hom}_{G_{\bbC}}(F_{\xi}^{\vee}, \bbC)$.
In view of the proof of Proposition \ref{prop-compatible}, we have $\phi_F=\phi_L\otimes \phi_R$, where $\phi_L$  is a nonzero element in  $ \mathrm{Hom}_{K}( E_{\mu^L}^{\vee}\otimes E_{\nu^L}^{\vee} \otimes E_{\lambda_L}, \bbC)$, and $\phi_R$ is a nonzero element in  $ \mathrm{Hom}_{K}( E_{\mu^R}\otimes E_{\nu^R} \otimes E_{\lambda_R}^{\vee}, \bbC)$.

We denote a nonzero highest (resp. lowest) weight vector of $E_{\lambda_L}$ by $\lambda_L^+$ (resp. $\lambda_L^-$). Note that as a $K$-type, $E_{\lambda_L}^{\vee}$ can be realized on the same space as $E_{\lambda_L}$, and then  $\lambda_L^+$ (resp. $\lambda_L^-$) becomes a lowest (resp. highest) weight vector. We interpret $\mu_L^+$, $\mu_L^{-}$, $\nu_L^+$ and $\nu_L^-$ similarly.

Denote all the lower triangular matrices in $\frg\frl_n:=\frg\frl(n, \bbC)$  with real diagonal entries by $\frb_1$, and denote those upper triangular matrices with purely imaginary diagonal entries by $\frb_2$. We consider the  subalgebra $\frb:=\frb_1\oplus \frb_2 \oplus \frg\frl_n$ of $\frg\frl_n \oplus\frg\frl_n \oplus\frg\frl_n$. Let $\frh$ be the diagonal embedding of $\frg\frl_n$ in $\frg\frl_n \oplus\frg\frl_n \oplus\frg\frl_n$. Then it is direct to check that $\frh\cap\frb=\{0\}$. Then by dimension counting, we have
$\frg\frl_n \oplus\frg\frl_n \oplus\frg\frl_n=\frh\oplus \frb$.
Therefore, $$E_{\mu_L}^{\vee}\otimes E_{\nu_L}^{\vee}\otimes E_{\lambda_L}=U(\frh)U(\frb) (\mu_L^+\otimes \nu_L^-\otimes \lambda_L^+)=U(\frh) (\mu_L^+\otimes \nu_L^-\otimes E_{\lambda_L}).$$
Thus $\phi_L$ does not vanish on  the space $\mu_L^+\otimes \nu_L^-\otimes E_{\lambda_L}$.
Similarly, $\phi_R$ does not vanish on  the space $\mu_R^-\otimes \nu_R^+\otimes E_{\lambda_R}^{\vee}$.

Denote by $(E_{\lambda_L})_{\lambda}$ the $\lambda$ weight space of $E_{\lambda_L}$.
Now let us adopt the setting of  Proposition \ref{prop-BW} (c), and prove the second assertion. The case of Proposition \ref{prop-BW} (b) is similar.  Note that each weight space of $E_{\lambda_L}$ and $E_{\lambda_R}^{\vee}$ is one dimensional. We choose a basis of $E_{\lambda_L}$ (resp. $E_{\lambda_R}^{\vee}$) consisting of weight vectors.
Since $\phi_L$ maps $E_{\mu^L}^{\vee}\otimes E_{\nu^L}^{\vee} \otimes E_{\lambda_L}$ onto the trivial representation $K$-equivariantly, we have that $\phi_L$ does not vanish on $\mu_L^+\otimes \nu_L^-\otimes (E_{\lambda_L})_{\lambda}$ if and only if the latter space has weight zero. That is, if and only if  $\lambda=\mu_L+ w_0\nu_L$. We put $\lambda_L^0:=\mu_L+ w_0\nu_L$, and put $a_k:=\mu_k+\nu_{n+1-k}+j-1$ for $2\leq k\leq n$. Note that each $a_k$ is nonnegative since $j$ is assumed to be a critical place for $\pi_{\mu}\times \pi_{\nu}$,  see \eqref{cond-c}.
Now we have
$$\lambda_L-\lambda_L^0=\sum_{k=2}^{n} a_k(\epsilon_1 -\epsilon_k),$$
and there exists an element $E_L\in U(\frg\frl_n)_{-\lambda_L+\lambda_L^0}$ such that
\be\label{L}
E_L \, . \,\lambda_L^+=w_L,
\ee
where $w_L$ is a nonzero vector in $(E_{\lambda_L})_{\lambda_L^0}$, and it is among the chosen basis of $E_{\lambda_L}$. Here ``." means the action of $\frg\frl_n$ on $E_{\lambda_L}$.

Similarly, $\phi_R$ does not vanish on $\mu_R^-\otimes \nu_R^+\otimes (E_{\lambda_R}^{\vee})_{\lambda}$ if and only if $\lambda=- w_0 \mu_R - \nu_R$. We put $\lambda_R^0:=- w_0 \mu_R - \nu_R$, and put $b_k:=-j-2\kappa+\mu_k+\nu_{n+1-k}$ for $2\leq k\leq n$, which are all nonnegative by \eqref{cond-c}.
Note that
$$-w_0\lambda_R-\lambda_R^0=\sum_{k=2}^{n} b_k(\epsilon_1 -\epsilon_k),$$
and there exists an element $E_R\in U(\frg\frl_n)_{w_0\lambda_R + \lambda_R^0}$ such that
\be\label{R}
E_R \, .\, \lambda_R^-=w_R,
\ee
where $w_R$ is a nonzero vector in $(E_{\lambda_R}^{\vee})_{\lambda_R^0}$, and it is among the chosen basis of $E_{\lambda_R}^{\vee}$.

The vector $\lambda_L^+ \otimes \lambda_R^-$ is a highest weight vector of the Cartan component of $E_{\lambda_L}\otimes E_{\lambda_R}^{\vee}$. Moreover, by \eqref{L} and \eqref{R}, we have
\be\label{key-vector}
w_0:= E_R E_L \, . \,(\lambda_L^+\otimes \lambda_R^-)=w_L\otimes w_R + \mbox{other terms.}
\ee
Here each linear component in ``other terms" is different from $w_L\otimes w_R$, and has the form $w_1\otimes w_2$, where $w_1$ (resp. $w_2$) is among the chosen basis of $E_{\lambda_L}$ (resp. $E_{\lambda_R}^{\vee}$), and $w_1\otimes w_2 \in (E_{\lambda_L}\otimes E_{\lambda_R}^{\vee})_{\lambda_L^0+\lambda_R^0}$. We necessarily have $\phi_L(\mu_L^+\otimes \nu_L^-\otimes w_1)=0$ or $\phi_R(\mu_R^-\otimes \nu_R^+\otimes w_2)=0$. Therefore,
$$
\phi_F((\mu_L^+ \boxtimes \mu_R^-)  \otimes (\nu_L^-\boxtimes \nu_R^+)   \otimes
(w_1\boxtimes w_2))=\phi_L(\mu_L^+\otimes \nu_L^-\otimes w_1) \phi_R(\mu_R^-\otimes \nu_R^+\otimes w_2)=0.
$$
Thus by \eqref{key-vector},
$$
\phi_F((\mu_L^+ \boxtimes \mu_R^-)  \otimes (\nu_L^-\boxtimes \nu_R^+)   \otimes
w_0)=\phi_L(\mu_L^+\otimes \nu_L^-\otimes w_L) \phi_R(\mu_R^-\otimes \nu_R^+\otimes w_R)\neq 0.
$$
Note that $\mu_L^+ \otimes \mu_R^-$ is a lowest weight vector of $\tau_{\mu}^{\vee}$, and that $\nu_L^- \otimes \nu_R^+$ is a highest weight vector of $\tau_{\nu}^{\vee}$.
Moreover, the vector $w_0$ defined in \eqref{key-vector} lives in $\sigma_j$, the Cartan component of $E_{\lambda_L}\otimes E_{\lambda_R}^{\vee}$. We conclude that $\phi_F$
does not vanish on $\tau_{\xi}^{\vee}=\tau_{\mu}^{\vee}\otimes \tau_{\nu}^{\vee}\otimes \sigma_j$. This finishes the proof.
\ep

\subsection{Analysis of linear functionals on PRV components}
The following lemma is taken from  Section 2.1 of \cite{Ya}.

\begin{lemma}\label{lemma-Ya}
Let $\alpha_1$ and $\alpha_2$ be two irreducible  representations of a compact connected Lie group $L$. Let $\alpha_3$ be the Cartan component of $\alpha_1\otimes \alpha_2$.
Let $f : \alpha_1 \otimes \alpha_2 \to \alpha_3$ be a nonzero $L$-equivariant linear map. Then $f$
maps all nonzero decomposable vectors (namely, vectors of the form $u\otimes v\in \alpha_1\otimes \alpha_2$) to nonzero vectors.
\end{lemma}

Let $\sigma_1$ and $\sigma_2$ be two irreducible representations of $K^3$, and write $\sigma_3$ for their Cartan component. The remaining part of this section is devoted to deducing an analog of Proposition 2.16 of \cite{Sun}. Since our situation is slightly more complicated,  we give a proof here.

\begin{prop}\label{prop-PRV}
Assume that $\dim\,(\sigma_i)^{K}=1$ for $1\leq i\leq  3$. Let $\phi_1$ (resp. $\phi_3$) be any nonzero $K$-invariant linear functional  on $\sigma_1^{\vee}$ (resp. $\sigma_3$). Then $\phi_1\otimes \phi_3$ does not vanish on the PRV component
$$
\sigma_2 \subset \sigma_1^{\vee}\otimes \sigma_3.
$$
\end{prop}
\bp
 Let $\sigma_i=\alpha_i\otimes \beta_i\otimes \gamma_i$ for $1\leq i\leq 3$. Fix a generator $v_2\in (\alpha_2\otimes \beta_2\otimes \gamma_2)^K$. It is routine to check that the following diagram commutes:
\[
\begin{CD}
\mathrm{Hom}_{K^3}(\alpha_2\otimes \beta_2\otimes \gamma_2, \alpha_1^{\vee}\otimes \beta_1^{\vee}\otimes\gamma_1^{\vee}\otimes \alpha_3\otimes \beta_3\otimes \gamma_3) @> f\mapsto ((\phi_1\otimes\phi_3)\circ f)(v_2) >\phantom{f\otimes g\otimes h \mapsto \phi_3(h\circ (\phi_1\otimes v_2)\circ (f\otimes g))}> \bbC \\
@VV \cong V  @V=VV \\
\mathrm{Hom}_{K}(\gamma_3^{\vee}, \gamma_1^{\vee}\otimes \gamma_2^{\vee})
\otimes \mathrm{Hom}_{K}(\beta_3^{\vee}, \beta_1^{\vee}\otimes \beta_2^{\vee})
\otimes
\mathrm{Hom}_{K}(\alpha_1\otimes \alpha_2, \alpha_3) @>f\otimes g\otimes h \mapsto \phi_3(h\circ (\phi_1\otimes v_2)\circ (f\otimes g))>> \bbC
\end{CD}
\]
where the left vertical arrow is the canonical isomorphism, and in the bottom
horizontal arrow, we view
\begin{eqnarray*}
v_2\in \mathrm{Hom}_K(\beta_2^{\vee}\otimes \gamma_2^{\vee}, \alpha_2)&=& (\alpha_2\otimes \beta_2\otimes \gamma_2)^K,\\
\phi_1\in \mathrm{Hom}_{K}(\beta_1^{\vee}\otimes \gamma_1^{\vee}, \alpha_1)&=& \mathrm{Hom}_{K}(\alpha_1^\vee\otimes \beta_1^\vee\otimes \gamma_1^\vee, \bbC), \mbox{ and}
\end{eqnarray*}
$$
h\circ (\phi_1\otimes v_2)\circ (f\otimes g)\in \mathrm{Hom}_{K}(\gamma_3^{\vee}\otimes \beta_3^{\vee}, \alpha_3) = (\alpha_3\otimes \beta_3\otimes \gamma_3)^K.
$$

Now it suffices to check that the bottom horizontal arrow of the diagram is nonzero. We pick up a generator
$$
f_0\otimes g_0\otimes h_0\in \mathrm{Hom}_{K}(\gamma_3^{\vee}, \gamma_1^{\vee}\otimes \gamma_2^{\vee})
\otimes \mathrm{Hom}_{K}(\beta_3^{\vee}, \beta_1^{\vee}\otimes \beta_2^{\vee})
\otimes
\mathrm{Hom}_{K}(\alpha_1\otimes \alpha_2, \alpha_3).
$$
Note that $\beta_3^{\vee}$ (resp. $\gamma_3^{\vee}$) is the Cartan component of $\beta_1^{\vee}\otimes \beta_2^{\vee}$ (resp. $\gamma_1^{\vee}\otimes \gamma_2^{\vee}$).
Let $u_3$ be a nonzero highest  weight vector of $\beta_3^{\vee}$, and let $v_3$ be a nonzero lowest weight vector of $\gamma_3^{\vee}$. By Lemma 2.11 of \cite{Sun}, $g_0(u_3)$ (resp. $f_0(v_3)$) is a nonzero decomposable vector in   $\beta_1^{\vee}\otimes \beta_2^{\vee}$ (resp. $\gamma_1^{\vee}\otimes \gamma_2^{\vee}$). Therefore, $(\phi_1\otimes v_2)(f_0\otimes g_0)(v_3\otimes u_3)$ is a decomposable vector in $\alpha_1\otimes \alpha_2$. We will soon show that it is actually nonzero.
Now Lemma \ref{lemma-Ya} implies that
$$
(h_0\circ (\phi_1\otimes v_2)\circ(f_0\otimes g_0))(v_3\otimes u_3)\neq 0.
$$
Thus $h_0\circ (\phi_1\otimes v_2)\circ(f_0\otimes g_0)$ is a nonzero generator of the one dimensional space
$$
\mathrm{Hom}_{K}(\gamma_3^{\vee}\otimes \beta_3^{\vee}, \alpha_3)=(\alpha_3\otimes \beta_3\otimes \gamma_3)^K.
$$
Since $\phi_3$ does not vanish on $(\alpha_3\otimes \beta_3\otimes \gamma_3)^K$, this shows that the bottom horizontal arrow is nonzero, as desired.

It remains to show that $(\phi_1\otimes v_2)(f_0\otimes g_0)(v_3\otimes u_3)$ is nonzero.
For $i=1, 2$, let $\mu_i$ be the highest weight of $\beta_i$, and let $\nu_i$ be the highest weight of $\gamma_i$; denote by $\mu_i^+$ (resp. $\mu_i^{-}$) a nonzero highest (resp. lowest) weight vector of $\beta_i$; denote by $\nu_i^+$ (resp. $\nu_i^{-}$) a nonzero highest (resp. lowest) weight vector of $\gamma_i$. We realize $\beta_1^{\vee}$ on the same space of $\beta_1$. Then, say, $\mu_1^+$ is a lowest weight vector of $\beta_1^{\vee}$.
Up to nonzero scalars,
$$g_0(u_3)=\mu_1^-\otimes \mu_2^-, \quad f_0(v_3)=\nu_1^+\otimes \nu_2^+.$$
Using the same technique as in Lemma \ref{lemma-mult-one-sigma-j}, one has that $\phi_1(\mu_1^-\otimes \nu_1^+\otimes \alpha_1^\vee)\neq 0$. Moreover, then one deduces that
$\phi_1(\mu_1^-\otimes \nu_1^+\otimes (\alpha_1^\vee)_{\lambda})$ is nonzero if and only if $\lambda=w_0\mu_1+\nu_1$. Here $(\alpha_1^\vee)_{\lambda}$ denotes the $\lambda$ weight space of $\alpha_1^\vee$. Therefore, when $\phi_1$ is viewed as a $K$-equivariant homomorphism from $\beta_1^{\vee}\otimes \gamma_1^{\vee}$ to $\alpha_1$, we have that
$$
\phi_1(\mu_1^{-}\otimes \nu_1^+)\in (\alpha_1)_{-w_0\mu_1-\nu_1}
$$
is nonzero.
Similarly,
$$
v_2(\mu_2^{-}\otimes \nu_2^+)\in (\alpha_2)_{-w_0\mu_2-\nu_2}
$$
is nonzero.
Thus
$$
(\phi_1\otimes v_2)(f_0\otimes g_0)(v_3\otimes u_3)=\phi_1(\mu_1^{-}\otimes \nu_1^+)\otimes v_2(\mu_2^{-}\otimes \nu_2^+)
$$
is nonzero. This finishes the proof.
\ep

\section{Infinite dimensional representations}
Recall the pure weights $\mu$ and $\nu$ from \eqref{mu} and \eqref{nu}, respectively.
Assume that $\mu$ and $\nu$ are compatible. Furthermore, we adopt the assumption \eqref{assumption} from now on. That is, we assume that $\kappa$ is an half integer and fix $j=-\kappa+\frac{1}{2}$, which is a critical place for $\pi_{\mu}\times \pi_{\nu}$.
Recall the $K$-types $\tau_\mu^+$, $\tau_\nu^+$ and $\sigma_j^+$ from the previous section. Recall from \eqref{tau-xi-plus} that $\tau_\xi^+:=\tau_\mu^+\otimes\tau_\nu^+\otimes\sigma_j^+$.
By Lemmas \ref{lemma-Ij-min-K} and \ref{lemma-Jmu-min-K}, the $K^3$-representation $\tau_\xi^+$ occurs with multiplicity one in $\pi_\xi$.

Recall that $G=\GL(n, \K)$. Fix $H=\SL(n, \K)$. This section aims to prove the following proposition.

\begin{prop}\label{prop-not-vanish-tau-xi-plus} Under the assumption \eqref{assumption}, every nonzero element of $\mathrm{Hom}_{G}(\pi_{\xi}, \bbC)$ does not vanish on $\tau_{\xi}^{+}\subset\pi_{\xi}$.
\end{prop}

By Section \ref{sec-intro}, the representations $\pi_\mu\vert_{H}$ and $\pi_\nu\vert_{H}$ are unitarizable and tempered. Also, $I_j\vert_{H}$ is unitarizable under the
assumption \eqref{assumption}. We fix a $H^3$-invariant positive definite continuous Hermitian form $\langle\, , \,\rangle_\xi:=\langle\, , \,\rangle_{\pi_\mu}\otimes\langle\, , \,\rangle_{\pi_\nu}\otimes\langle\, , \,\rangle_{I_j}$ on $\pi_\xi$.

\begin{lemma}\label{lemma-hermitian-form}
The integrals in
\be \label{integral}
\begin{array}{rcl}
\pi_{\xi}\times\pi_{\xi} & \to & \C,\\
(u,v) & \mapsto &  \int_H\langle h.u, v \rangle_\xi \od\! h
\end{array}
\ee
converge absolutely for all $u,v\in \pi_{\xi}$, and yields a continuous $H$-invariant Hermitian form on $\pi_{\xi}$. Here ``$\od\! h$'' denotes a Haar measure on $H$.
\end{lemma}
\bp
We may and do assume that
\[
u=u_1\otimes u_2\otimes u_3 \quad \mbox{and} \quad v=v_1\otimes v_2 \otimes v_3,
\]
where $u_1, v_1\in \pi_\mu$, $u_2, v_2 \in \pi_\nu$ and $u_3, v_3 \in I_j$. Let $H=K_0A^+K_0$ be the Cartan decomposition of $H$, where
\[
A^+ :=\{\mbox{diag}(a_1,a_2,\dots,a_n)\vert \prod_{i=1}^n a_i=1,\mbox{ and } a_i\geq a_{i+1}>0\mbox{ for all }1\leq i\leq n-1\}
\]
is the positive Weyl chamber, and $K_0=\SU(n)$ is the maximal compact subgroup of $H$.
It is well known that $\int_H\vert\langle h.u, v \rangle_\xi\vert \od\! h<\infty$ if and only if $\int_{A^+}\vert\langle a.u, v \rangle_\xi\vert \delta_B(a)\od\!^\times a<\infty$, where $\od\!^\times a$ is a Haar measure on $A^+$, and $\delta_B$ is the modular character of the standard Borel subgroup $B$ of $H$.

Denote by $\Xi_{K_0}$ the Harish-Chandra function on $H$ associated to the maximal compact subgroup $K_0$ (see \cite[Section 4.5.3]{Wa1}). By \cite[Theorem 1.2]{Sun09-2}, there is a continuous seminorm $\vert \, \cdot\,\vert_{\pi_\mu}$ on $\pi_\mu$ such that
\be \label{inequal-1}
\vert\langle \pi_\mu(a).u_1, v_1\rangle_{\pi_\mu}\vert \leq \Xi_{K_0}(a)\cdot\vert u_1\vert_{\pi_\mu} \vert v_1\vert _{\pi_\mu}, \quad u_1,v_1\in \pi_\mu, a \in A^+.
\ee
Analogously, there is a continuous seminorm $\vert \, \cdot\,\vert_{\pi_\nu}$ on $\pi_\nu$ such that
\be \label{inequal-2}
\vert\langle \pi_\nu(a).u_2, v_2\rangle_{\pi_\nu}\vert \leq \Xi_{K_0}(a)\cdot\vert u_2\vert_{\pi_\nu} \vert v_2\vert _{\pi_\nu}, \quad u_2,v_2\in \pi_\nu, a \in A^+.
\ee
Applying Langlands classification (due to Zhelobenko in the case of complex groups \cite{Zh}), one sees that $I_j$ is isomorphic to the irreducible quotient of $\Ind_B^H (\chi \otimes 1)$, where $B=LN$ is the Borel subgroup with Levi factor $L$ and unipotent radical $N$, and
\[
\chi(a)=(a_1\overline{a_1})^{\frac{n-2}{2}-\kappa}(a_2\overline{a_2})^{\frac{n-4}{2}-\kappa}\cdots(a_{n-1}\overline{a_{n-1}})^{-\frac{n-2}{2}-\kappa}a_n^{(n-1)\kappa-k_\eta}\overline{a_n}^{k_\eta-(n+1)\kappa}
\]
for $a=\mbox{diag}(a_1,a_2,\dots,a_n)\in L$. By \cite[Lemma 5.2.8]{Wa1}, there are constants $c,d>0$ such that
\be \label{inequal-3}
\vert\langle I_j(a).u_3, v_3\rangle_{I_j}\vert \leq c\cdot\chi(a)\Xi_{K_0}(a)(1+\log \| a \|)^d, \quad u_3,v_3\in I_j, a \in A^+,
\ee
where $\|\cdot\|$ is certain norm on $H$. By the estimate of $\Xi_{K_0}$ in \cite[Theorem 4.5.3]{Wa1} and \eqref{inequal-1}, \eqref{inequal-2}, \eqref{inequal-3}, the integral $\int_{A^+}\vert\langle a.u, v \rangle_\xi\vert \delta_B(a)\od\!^\times a$ is convergent. Thus $\int_H\langle h.u, v \rangle_\xi \od\! h$ converges absolutely.
The  map defined by \eqref{integral} yields an $H$-invariant Hermitian form since $H$ is unimodular.
\ep
Write $H=SK_0$ for the Cartan decomposition of $H$, where
\[
S:=\{ \mbox{positive definite Hermitian matrices with determint  1}\}.
\]
\begin{lemma}\label{lemma-positive-matrix-coef} \emph{(\cite[Theorem 1.4]{Sun09})}
For every nonzero vector $u$ in the minimal $K_0^3$-type $\tau_{\xi}^+$ of $\pi_{\xi}$, one has that
\[
\langle g.u, u\rangle _{\xi} > 0, \quad \mbox {for all } g\in S^3.
\]
\end{lemma}

By Lemma \ref{lemma-inv-dim-one-tau-xi-plus}, the space $(\tau_\xi^+)^{K_0}\neq 0$. Take a nonzero element $v_\xi \in(\tau_\xi^+)^{K_0}$. Then we have the following positivity lemma.
\begin{lemma}\label{lemma-positive-v-xi}
The Hermitian form defined in \eqref{integral} is positive definite on the one dimensional space $\C v_{\xi}$.
\end{lemma}
\bp
Denote by $\od\! k$ the Haar measure on $K_0$ such that the volume of $K_0$ is $1$, and let $\od\! s$ be a certain positive measure on $S$. By Lemma \ref{lemma-positive-matrix-coef}, one has that
\[
\begin{split}
&\int_H\langle h.v_{\xi},v_{\xi}\rangle_{\xi}\od\! h \\
=&\int_S\int_{K_0}\langle sk.v_{\xi},v_{\xi}\rangle_{\xi}\od\! k\od\! s \\
=& \int_S\langle s.v_{\xi},v_{\xi}\rangle_{\xi}\od\! s > 0.
\end{split}
\]
\ep

\begin{lemma}\label{lemma-not-vanish-tau-xi-plus} Under the assumption \eqref{assumption}, there exists an  element of $\mathrm{Hom}_{G}(\pi_{\xi}, \bbC)$ which does not vanish on $\tau_{\xi}^{+}\subset\pi_{\xi}$.
\end{lemma}

\bp
Define
\[
\begin{array}{rcl}
\phi : \pi_\xi &\to &\C, \\
u & \mapsto &\int_H\langle h.u, v_\xi \rangle_\xi \od\! h.
\end{array}
\]
Then $\phi\in \Hom_H(\pi_{\xi}, \bbC)$. By Lemma \ref{lemma-positive-v-xi}, $\phi$ does not vanish on $\tau_{\xi}^{+}$. Since the representation $\pi_{\xi}$ has a trivial central character when restricted to $G$, the map
$\phi$ is $G$-intertwining, which completes the proof of the lemma.
\ep

\begin{lemma}\label{lemma-mult-one} Suppose that $\mu$ and $\nu$ are compatible (Definition \ref{def-compatible}), and $j$ is a critical place for $\pi_\mu \times \pi_\nu$. Then the inequality
\be \label{multi-1}
\mathrm{dim \mspace{6mu} Hom}_G(\pi_\xi, \bbC)\leq 1
\ee
holds.
\end{lemma}

\bp
Denote by $\CS(\C^n)$ the space of Schwartz functions on $\C^n$. The space $\CS(\C^n)$ carries an action of $G$ defined by
\[
g.f(v):=f(vg), \quad v\in \C^n, f\in\CS(\C^n), g\in G.
\]
Let $Z\cong \C^\times$ be the center of $G$, and let $P_0$ be the subgroup of $P$ with the last row being $v_0:=(0,\dots, 0 ,1 )$. Then $P=P_0\cdot Z$. Define a character $\chi_j$ of $Z$ as
\[
\chi_j(z):=z^{-nj-k_\eta}\cdot \overline z ^ {-nj+k_\eta-2n\kappa}, \quad z\in Z.
\]
Then
\[
I_j=\uInd_P^G H_j=  \vert \det\vert_\C^j \otimes \uInd_P^G(1\otimes \chi_j),
\]
where $\mspace{1mu}^\mathrm{u}\mspace{-3mu}\Ind$ stands for the non-normalized smooth induction.

Write $\CS(\C^n\setminus\{0\})$ for the space of Schwartz functions on $\C^n\setminus\{0\}$. The linear map
\[
\begin{array}{rcl}
\CS(\C^n\setminus\{0\}) &\to & C^\infty(G),\\
\varphi & \mapsto & \left(g\mapsto\int_Z \varphi(v_0gz)\chi_j(z)^{-1}\od\! h\right)
\end{array}
\]
induces a $G$-intertwining isomorphism
\be
\CS(\C^n\setminus\{0\})_{Z,\,\chi_j}\cong \Ind_P^G(1\otimes \chi_j),
\ee
where $(\,\cdot\,)_{Z,\,\chi_j}$ indicates the maximal Hausdorff quotient space on which $Z$ acts through the character $\chi_j$. Suppose we are in case (b) of Proposition \ref{prop-BW}, then
\[
-nj-k_\eta \geq 0 \quad \mbox{and} \quad -nj+k_\eta-2n\kappa\leq -n.
\]
By Theorem 1.1 (a) of \cite{Xue}, the natural imbedding $\CS(\C^n\setminus\{0\})\hookrightarrow \CS(\C^n)$ induces an isomorphism
\be
\CS(\C^n\setminus\{0\})_{Z,\,\chi_j}\cong \CS(\C^n)_{Z,\,\chi_j}
\ee
of representations of $G$. Hence
\be \label{iso-deg-to-schwartz}
I_j\cong \vert \det\vert_\C^j \otimes \CS(\C^n\setminus\{0\})_{Z,\,\chi_j}\cong \vert \det\vert_\C^j \otimes \CS(\C^n)_{Z,\,\chi_j}.
\ee
By \cite[Theorem C]{SZ}, for all irreducible Casselman-Wallach representations $\pi_1, \pi_2$ and every character $\chi$ of $G$, one has that
\be \label{dim-leq-1}
\mathrm{dim \mspace{6mu} Hom}_G(\pi_1\otimes \pi_2\otimes \CS(\C^n), \chi)\leq 1.
\ee
Now the inequality \eqref{multi-1} is a straightforward consequence of  \eqref{iso-deg-to-schwartz} and \eqref{dim-leq-1}. Analogous analysis implies that \eqref{multi-1} also holds for case (c) of Proposition \ref{prop-BW}.

\ep

By Lemmas \ref{lemma-not-vanish-tau-xi-plus} and \ref{lemma-mult-one}, it is easy to see that Proposition \ref{prop-not-vanish-tau-xi-plus} holds.

\section{The relative cohomology spaces}

Note that the center $\K^{\times}$ of $GL(n, \K)$ acts trivially on $\pi_\mu\otimes F_\mu^{\vee}$. Recall that $\frg=\mathfrak{g}\mathfrak{l}_n(\bbC)\times \mathfrak{g}\mathfrak{l}_n(\bbC)$, and that $\tilde{\frk}$ is the complexified Lie algebra of $\tilde{K}=GU(n)$. Now we have
\begin{eqnarray*}
\mathrm{H}^{b_n}(\frg, GU(n); \pi_\mu\otimes F_\mu^{\vee}) &=& \mathrm{H}^{b_n}(\mathfrak{s}\mathfrak{l}_n(\bbC)\times \mathfrak{s}\mathfrak{l}_n(\bbC), SU(n); \pi_\mu\otimes F_\mu^{\vee})\\
&=& \mathrm{Hom}_{SU(n)}(\wedge^{b_n}(\mathfrak{s}\mathfrak{l}_n(\bbC)\times \mathfrak{s}\mathfrak{l}_n(\bbC)/\mathfrak{s}\mathfrak{l}_n(\bbC)); \pi_\mu\otimes F_\mu^{\vee})\\
&=& \mathrm{Hom}_{SU(n)}(\wedge^{b_n}(\frg/\tilde{\frk}); \pi_\mu\otimes F_\mu^{\vee}),
\end{eqnarray*}
where at the penultimate step we use Proposition 9.4.3 of \cite{Wa1}. Note that $\pi_{\mu}|_{SL_n(\K)}$ is unitary (actually tempered). The analog also holds for $\pi_\nu\otimes F_\nu^{\vee}$. Now assume \eqref{assumption}. Then $j=-\kappa+\frac{1}{2}$, note that $I_{j}|_{SL_n(\K)}$ is unitary and $I_j\otimes V_j$ has a trivial central character as well. Thus, we have
$$
\mathrm{H}^{c_n}(\frg, GU(n); I_j\otimes V_j) = \mathrm{Hom}_{SU(n)}(\wedge^{c_n}(\frg/\tilde{\frk}); I_j\otimes V_j).
$$

\section{Proof of Theorem A}
Recall that $G=\GL(n, \K)$ and $H=\SL(n, \K)$.
Let us continue to assume \eqref{assumption}.  Now we are ready to prove Theorem A. We only work with Proposition \ref{prop-BW} (b).
Indeed, by the discussion in the previous section, we have that
\be\label{relative-coho-big}
\rmH^{2b_n+c_n}\left(\frg^3, \tilde K^3; \pi_\xi \otimes F_\xi^{\vee}\right)=\mathrm{Hom}_{SU(n)}\left(\wedge^{2b_n+c_n}\left((\frg/\tilde{\frk})^3\right); \pi_{\xi}\otimes F_{\xi}^{\vee}\right).
\ee
Here we write $\frg/\tilde{\frk}\oplus \frg/\tilde{\frk}\oplus \frg/\tilde{\frk}$ as $(\frg/\tilde{\frk})^3$ for short.
Likewise,
\be\label{relative-coho-small}
\rmH^{2b_n+c_n}(\frg, \tilde K; \bbC)=\mathrm{Hom}_{SU(n)}(\wedge^{2b_n+c_n}(\frg/\tilde{\frk}); \bbC).
\ee

Note that $\tau_n\otimes\tau_n\otimes \sigma_n$ is the PRV-component of $\tau_{\xi}^{+}\otimes\tau_{\xi}^{\vee}$. Write
$$
\varphi_{\xi}:  \tau_n\otimes\tau_n\otimes \sigma_n \to \pi_\xi\otimes F_\xi^{\vee}
$$
for the inclusion
$$
\tau_n\otimes\tau_n\otimes \sigma_n \subset \tau_{\xi}^{+}\otimes\tau_{\xi}^{\vee}
\subset \pi_\xi\otimes F_\xi^{\vee}.
$$

Recall from the introduction that we can pick up a nonzero element
$$
\phi_\pi \in \Hom_G(\pi_\xi, \bbC)
$$
which does not vanish on $\tau_{\xi}^+$, see Proposition \ref{prop-not-vanish-tau-xi-plus}.
On the other hand, Lemma \ref{lemma-mult-one-sigma-j} guarantees the existence of a nonzero element
$$
\phi_F\in\mathrm{Hom}_{G_{\bbC}}(F_{\xi}^{\vee}, \bbC)
$$
which does not vanish on $\tau_\xi^\vee$. Recall the map $\eta_n:
\wedge^{2b_n+c_n}((\frg/\tilde{\frk})^3) \to \tau_n\otimes\tau_n\otimes\sigma_n$ from \eqref{composition-iota-eta}. The composition $\varphi_\xi\circ \eta_n$ is an element of
\eqref{relative-coho-big}. Its image under the map \eqref{map-ThmA} of Theorem A equals the composition map
\be\label{composition-final}
\begin{CD}
\wedge^{2b_n+c_n}(\frg/\tilde{\frk}) @> \iota_n > > \wedge^{2b_n+c_n}((\frg/\tilde{\frk})^3)  @> \eta_n > > \tau_n\otimes\tau_n\otimes\sigma_n  @> \varphi_\xi > >
\pi_\xi\otimes F_\xi^{\vee} @> \phi_\pi\otimes \phi_F > > \bbC.
\end{CD}
\ee
By Proposition \ref{prop-PRV}, the composition $(\phi_\pi\otimes \phi_F) \circ \varphi_\xi$ is nonzero. Since it is $K$-invariant, it does not vanish on $(\tau_n\otimes\tau_n\otimes\sigma_n)^K$. By Lemma \ref{lemma-mult-one-tau-n-tau-n-sigma-n}, the latter space is equal to the image of $\eta_n\circ \iota_n$. Therefore the composition \eqref{composition-final} is nonzero. This finishes the proof of Theorem A.




%

\medskip

\centerline{\scshape Acknowledgements} We are heartily grateful to Binyong Sun  for guiding us through this project. We thank Anthony Knapp  for sharing knowledge with us. Part of this work is carried out during Dong's visit at MIT. He thanks the math department of MIT for offering excellent working conditions. He also thanks David Vogan sincerely for teaching him translation functors with great patience.

\end{document}